\documentclass[10pt]{amsart}

\usepackage{ amsmath,amssymb,verbatim,graphicx,amsthm,color}
\usepackage[top=1in, bottom=1in, left=1in, right=1in]{geometry}
\usepackage{cite}
 \allowdisplaybreaks[1]
\usepackage[sc]{mathpazo}
\usepackage[T1]{fontenc}

\newcommand{\vertiii}[1]{{\left\vert\kern-0.25ex\left\vert\kern-0.25ex\left\vert #1 
    \right\vert\kern-0.25ex\right\vert\kern-0.25ex\right\vert}}
\newcommand{\norm}[1]{\left\lVert#1\right\rVert}

\newcommand{\RR}{\mathbb{R}}
\newcommand{\CC}{\mathbb{C}}
\newcommand{\ZZ}{\mathbb{Z}}
\newcommand{\QQ}{\mathbb{Q}}
\newcommand{\NN}{\mathbb{N}}
\newcommand{\TT}{\mathbb{T}}

\newcommand{\expec}{\mathbb{E}}

\theoremstyle{plain}

\newtheorem{theorem}{Theorem}[section]

\newtheorem{lemma}[theorem]{Lemma}
\newtheorem{cor}[theorem]{Corollary}

\usepackage{setspace}
\begin{document}

\title{Extension of Wiener-Wintner double recurrence theorem to polynomials}

\author{Idris Assani}
\address{Department of Mathematics, The University of North Carolina at Chapel Hill, 
Chapel Hill, NC 27599}
\email{assani@math.unc.edu}
\urladdr{http://www.unc.edu/math/Faculty/assani/} 

\author{Ryo Moore}
\address{Department of Mathematics, The University of North Carolina at Chapel Hill, 
Chapel Hill, NC 27599}
\email{ryom@live.unc.edu}
\urladdr{http://ryom.web.unc.edu} 

\begin{abstract}
We extend our result on the convergence of double recurrence Wiener-Wintner averages to the case where we have a polynomial exponent. We will show that there exists a single set of full measure for which the averages 
\[ \frac{1}{N} \sum_{n=1}^N f_1(T^{an}x)f_2(T^{bn}x)\phi(p(n)) \]
converge for any polynomial $p$ with real coefficients, and any continuous function $\phi: \TT \to \CC$. We also show that if either function belongs to an orthogonal complement of an appropriate Host-Kra-Ziegler factor that depends on the degree of the polynomial $p$, then the averages converge to zero uniformly for all polynomials. This paper combines the authors' previously announced work.
\end{abstract}

\maketitle

\section*{Notations/Conventions}
Unless specified otherwise, the following notations/conventions will be used throughout the paper.
\begin{itemize}
	\item $\RR_k[\xi]$ is a collection of degree-$k$ polynomials with real coefficients, whereas $\RR[\xi]$ is the collection of all the polynomials with real coefficients.
	\item$\mathcal{Z}_k$ is the $k$-th Host-Kra-Ziegler factor (cf. \cite{HostKraNEA, Ziegler}), whereas $\vertiii{\cdot}_{k+1}$ is the Gowers-Host-Kra seminorm (cf. \cite{Gowers, HostKraNEA}) that characterizes the factor $\mathcal{Z}_k$.
	\item Sometimes, we denote $e(\alpha) = e^{2\pi i \alpha}$.
	\item Functions $f_1$ and $f_2$ are taken to be real-valued. The results presented here can be easily extended to the case where they are complex-valued functions.
	\item When we have $A \lesssim_{B} C$, we mean that $A \leq DC$ for some real constant $D > 0$ that depends on $B$.
\end{itemize}

\section{Introduction}
The following extension on Bourgain's pointwise result on double recurrence \cite{BourgainDR} was proven in \cite{WWDR}.

\begin{theorem}\label{WWDR}
	Let $(X, \mathcal{F}, \mu, T)$ be a standard ergodic dynamical system, $a, b \in \ZZ$ such that $a \neq b$, and $f_1, f_2 \in L^\infty(\mu)$. Let 
	\[W_N(f_1, f_2,x ,t) = \frac{1}{N} \sum_{n=1}^{N} f_1(T^{an}x)f_2(T^{bn}x) e^{2\pi i n t}. \]
	\begin{enumerate}
		\item (Double Uniform Wiener-Wintner Theorem) If either $f_1$ or $f_2$ belongs to $\mathcal{Z}_2^\perp$, then there exists a set of full measure $X_{f_1, f_2}$ such that for all $x \in X_{f_1, f_2}$,
		\begin{equation*}\label{UniformDWW} \limsup_{N \to \infty} \sup_{t \in \RR} \left|W_N(f_1, f_2, x, t) \right| = 0. \end{equation*}
		\item (General Convergence) If $f_1, f_2 \in \mathcal{Z}_2$, then for $\mu$-a.e. $x \in X$, $W_N(f_1, f_2, x, t)$ converges for all $t \in \RR$.
	\end{enumerate}
\end{theorem}
One of the estimates that was established to prove the uniform convergence result above was the following (this was obtained in the proof of \cite[Theorem 5.1]{WWDR}):
\begin{theorem}\label{Thm-estimate}
	Let $(X, \mathcal{F}, \mu, T)$ be a standard ergodic dynamical system, and $f_1, f_2 \in L^\infty(\mu)$. We have
	\begin{equation}\label{estimate}
	\int \limsup_{N \to \infty} \sup_{t \in \RR} \left| \frac{1}{N} \sum_{n=1}^N f_1(T^{an}x) f_2(T^{bn}x) e^{2\pi i n t} \right|^2 d\mu(x) \lesssim_{a, b} \min \left\{ \vertiii{f_1}_3^2\, , \, \vertiii{f_2}_3^2 \right\}.
	\end{equation}
\end{theorem}
In this paper, we will extend Theorem $\ref{WWDR}$ to the case where we have a polynomial exponent, i.e. we will show that there exists a set of full measure $X_{f_1, f_2}$ such that for any $x \in X_{f_1, f_2}$, the averages 
\[ \frac{1}{N} \sum_{n=1}^N f_1(T^{an}x)f_2(T^{bn}x)e(p(n)) \]
converge for all polynomials $p$ with real coefficients. Furthermore, we will show that an appropriate Host-Kra-Ziegler factor, depending on the degree of the polynomial $p$, is a characteristic factor for these averages. The statement is given precisely as follows:
\begin{theorem}\label{Main Theorem}
	Let $(X, \mathcal{F}, \mu, T)$ be a standard ergodic dynamical system, $a, b \in \ZZ$ such that $a \neq b$, and $f_1, f_2 \in L^\infty(\mu)$. Let 
	\begin{equation}\label{polyDRWWavg} W_N(f_1, f_2,x ,p) = \frac{1}{N} \sum_{n=1}^{N} f_1(T^{an}x)f_2(T^{bn}x) e(p(n)), \end{equation}
	where $p$ is a polynomial with real coefficients. Then the following are true:
	\begin{enumerate}
		\item If either $f_1$ or $f_2$ belongs to $\mathcal{Z}_{k+1}^\perp$, then there exists a set of full measure $X_{f_1, f_2}$ such that for all $x \in X_{f_1, f_2}$,
		\begin{equation*}\label{UniformDWW_allPoly} \limsup_{N \to \infty} \sup_{p \in \RR_k[\xi]} \left|W_N(f_1, f_2, x,p) \right| = 0. \end{equation*}
		\item If $f_1, f_2 \in \mathcal{Z}_{k+1}$, then for $\mu$-a.e. $x \in X$, the averages $W_N(f_1, f_2, x, p)$ converge for all $p \in \RR_k[\xi]$.
		\item There exists a set of full measure $X_{f_1, f_2}$ such that for all $x \in X_{f_1, f_2}$, the averages
		\[ \frac{1}{N} \sum_{n=1}^N f_1(T^{an}x)f_2(T^{bn}x)\phi(p(n))\]
		converge for all continuous functions $\phi: \TT \to \CC$ and polynomials $p$ with real coefficients.
	\end{enumerate}
\end{theorem}
We remark that when $a = b$, the averages in $(\ref{polyDRWWavg})$ become
\[ W_N(f_1, f_2, x, p) = \frac{1}{N} \sum_{n=1}^N (f_1 \cdot f_2)(T^{an}x)e^{2\pi i p(n)}. \]
which reduces to the polynomial extension of the pointwise ergodic theorem, and some results on a convergence of this type of averages have been achieved by E. Lesigne \cite{Lesigne90, Lesigne93} and N. Frantzikinakis \cite{Fran06}. Lesigne showed that for any ergodic system $(X, \mathcal{F}, \mu, T)$, there exists a set of full measure for which the averages
\begin{equation}\label{Lesigne} \frac{1}{N} \sum_{n=1}^N \phi(p(n)) f(T^nx) \end{equation}
converge for all polynomials $p$ and a continuous function $\phi:\TT \to \CC$. Furthermore, if $T$ is assumed to be totally ergodic, $p$ is a $k$-th degree polynomial, and $f$ belongs to the orthogonal complement of the $k$-th degree Abramov factor, then the averages in  $(\ref{Lesigne})$ converge to $0$. Frantzikinakis extended this result by showing the uniform counterpart: Assuming that $T$ is totally ergodic and $f$ belongs to the orthogonal complement of the $k$-th degree Abramov factor, then
\begin{equation*}
\lim_{N \to \infty} \sup_{p \in \RR_k[\xi]} \left|\frac{1}{N} \sum_{n=1}^N \phi(p(n)) f(T^nx)\right| = 0.
\end{equation*}
Frantzikinakis also showed that the assumption $T$ being totally ergodic cannot be replaced with $T$ merely ergodic by providing a counterexample.

The classical Wiener-Wintner averages are also generalized to the cases where $e(nt)$ is replaced by a nilsequence (cf. \cite{BHK_nilseq, HostKraUniformity}). For instance, B. Host and B. Kra showed \cite[Theorem 2.22]{HostKraUniformity} that given an ergodic system $(X, \mathcal{F}, \mu, T)$ and a function $f \in L^\infty(\mu)$, there exists a set of full measure $X_f \subset X$ such that for any $x \in X_f$, and for any nilsequence $(b_n)$, the averages
\[ \frac{1}{N}\sum_{n=1}^N f(T^nx) b_n \]
converge. Furthermore, the uniform version of this result was obtained by T. Eisner and P. Zorin-Kranich \cite[Theorem 1.2]{EZK2013}. We note that some aspects of polynomial Wiener-Wintner averages are covered in their studies because if $p$ is any degree-$k$ polynomial, then $b_n = e(p(n))$ is a $k$-step nilsequence.

Recently, T. Eisner and B. Krause obtained a uniform Wiener-Wintner results for averages with weights involving Hardy functions and for "twisted" polynomial ergodic averages \cite{EisnerKrause}.

\subsection{Remarks} This paper combines the authors' previous preprints that appeared during the summer of 2014 \cite{WWDR_nil, WWDR_poly2}. In this paper, we will focus on the stronger result obtained in \cite{WWDR_poly2} (which is Theorem $\ref{Main Theorem}$ of the current paper), and use that to prove the result obtained in \cite{WWDR_nil} (which is Corollary $\ref{WWDR_nil}$ of the current paper). It is worth mentioning that Corollary $\ref{WWDR_nil}$ was obtained before Theorem $\ref{Main Theorem}$ using different machineries.

{Since the first submission of this paper in September 2014, some new results have been announced. For instance, the authors showed that the sequence $a_n = f_1(T^{an}x)f_2(T^{bn}x)$ is a good universal weight for the Furstenberg averages in the $L^2$-norm for $\mu$-a.e. $x \in X$ \cite{NewUnivWeight_Norm}, which was extended further to the case where we have commuting transformations \cite{CommNorm}. The first author extended the double recurrence Wiener-Wintner result to nilsequences \cite{AssaniDRNil}, and a result similar to this was announced by P. Zorin-Kranich independently \cite{ZK_NilseqWW}.}
\section{Proof of Theorem $\ref{Main Theorem}$}
To prove Theorem $\ref{Main Theorem}$, we break the proof in two cases: The case where either $f_1$ or $f_2$ belongs to $\mathcal{Z}_{k+1}^\perp$, and the case where both $f_1$ and $f_2$ belong to $\mathcal{Z}_{k+1}$, where $k$ is the degree of the polynomial $p$. 

The first case corresponds exactly to (1) of Theorem $\ref{Main Theorem}$, and we prove this by applying induction on $k$. We notice that the base case $k = 1$ is essentially (1) of Theorem $\ref{WWDR}$, where $p(n) = tn$. For the inductive step, we first apply van der Corput's lemma to reduce the degree of the polynomial, which allows us to use the inductive hypothesis. Then we use the estimate similar to Theorem $\ref{Thm-estimate}$ for the polynomials of higher degree to control the integral of the limit supremum of these averages. Consequently, the uniform Wiener-Wintner result follows.

For the second case, we restrict ourselves to the case where both functions are measurable with respect to $\mathcal{Z}_{k+1}$. Using the fact that the $k+1$-th Host-Kra-Ziegler factor is an inverse limit of a sequence of $k+1$-step nilsystems that are factors of $(X, \mathcal{F}, \mu, T)$ (cf. \cite[Theorem 10.1]{HostKraNEA}), we further restrict ourselves to the case where $(X, \mathcal{F}, \mu, T)$ is an ergodic nilsystem. Furthermore, since the set of continuous functions is dense in $L^2(\mu)$, we will assume that $f_1$ and $f_2$ are continuous. These assumptions allow us to use Leibman's pointwise convergence result \cite{Leibman} to show that the averages
converge for all $x \in X$.

Combining these two cases, we prove (3) of Theorem $\ref{Main Theorem}$.

\subsection{Proof of (1) of Theorem $\ref{Main Theorem}$}
In this section, we prove (1) of Theorem $\ref{Main Theorem}$. One of the main inequalities used in this part of the proof is van der Corput's lemma, which is stated as follows (a proof can found in \cite{KuipersNiederreiter}):
\begin{lemma}[van der Corput] \label{lem-vdc}
	If $(a_n)$ is a sequence of complex numbers and if $H$ is an integer between $0$ and $N-1$, then
	\begin{align}\label{vdc}
	\left| \frac{1}{N} \sum_{n=0}^{N-1} a_n \right|^2 
	&\leq \frac{N+H}{N^2(H+1)} \sum_{n=0}^{N-1} |a_n|^2 \\
	&+ \frac{2(N+H)}{N^2(H+1)^2} \sum_{h=1}^H (H + 1 - h) Re \left( \sum_{n=0}^{N-h-1} a_n \overline{a}_{n+h} \right). \nonumber
	\end{align}
\end{lemma}
{The next lemma addresses the measurability of the map that takes a point in the phase space to the supremum of the polynomial Wiener-Wintner averages over a collection of polynomials with the same degree.
\begin{lemma}\label{measurablityLemma} For each positive integers $N$ and $k$, the map 
\[ x \in X \mapsto F_{N, k}(x) =  \sup_{p \in \RR_k[\xi]}\left| \frac{1}{N} \sum_{n=0}^{N-1} f_1(T^{an}x)f_2(T^{bn}x)e(p(n))\right|\]
is measurable.
\end{lemma}
\begin{proof}
If we denote $p(n) = \sum_{j=0}^k c_jn^j$, then
\begin{align*}
\sup_{p \in \RR_k[\xi]} \left| \frac{1}{N} \sum_{n=0}^{N-1} f_1(T^{an}x)f_2(T^{bn}x)e(p(n))\right|
&= \sup_{(c_0, c_1, c_2, \ldots, c_k) \in \RR^{k+1}}\left|\frac{1}{N} \sum_{n=0}^{N-1} f_1(T^{an}x)f_2(T^{bn}x)e\left(\sum_{j=0}^k c_jn^j\right) \right| \\
&= \sup_{(c_0, c_1, c_2, \ldots, c_k) \in \QQ^{k+1}}\left|\frac{1}{N} \sum_{n=0}^{N-1} f_1(T^{an}x)f_2(T^{bn}x)e\left(\sum_{j=0}^k c_jn^j\right) \right|,
\end{align*}
where the last equality follows from the fact that $\QQ^{k+1}$ is dense in $\RR^{k+1}$, and the map
\[ (c_0, c_1, c_2, \ldots, c_k) \mapsto e \left(\sum_{j=0}^k c_jn^j\right) \]
is a continuous one from $\RR^{k+1}$ to $\TT$ for each $n \in \ZZ$. Since $\QQ^{k+1}$ is countable, it follows that the map $x \mapsto F_{N, k}(x)$ is measurable for each $k$ and $N$. 
\end{proof}}
\subsubsection{Proof for the case $k = 2$} To better illustrate the proof of Theorem $\ref{Main Theorem}$(1), we first prove this result for the case where $p(n) = \alpha n^2 + \beta n$. Suppose that either $f_1$ or $f_2$ belongs to the orthogonal complement of $\mathcal{Z}_3$. Then we apply van der Corput's lemma and the Cauchy-Schwarz inequality to obtain
\begin{align*}
&\limsup_{N \to \infty}\sup_{\alpha, \beta \in \RR} \left| \frac{1}{N} \sum_{n=1}^N f_1(T^{an}x)f_2(T^{bn}x)e(\alpha n^2 + \beta n) \right|^2 \\
&\leq \frac{2}{H} + \frac{4}{(H+1)^2} \sum_{h=1}^H (H + 1 - h) \\
& \cdot \limsup_{N \to \infty} \sup_{\alpha, \beta \in \RR} Re \left( \frac{1}{N} \sum_{n=1}^N (f_1 \cdot f_1 \circ T^{ah})(T^{an}x) (f_2 \cdot f_2 \circ T^{bh})(T^{bn}x) e(-(\alpha h^2 + 2\alpha hn + \beta h)) \right) \\
&\leq \frac{2}{H} + 4\left(\frac{1}{H+1} \sum_{h=1}^H \limsup_{N \to \infty}\sup_{\alpha \in \RR} \left| \frac{1}{N} \sum_{n=1}^N (f_1 \cdot f_1 \circ T^{ah})(T^{an}x) (f_2 \cdot f_2 \circ T^{bh})(T^{bn}x) e(-2\alpha hn) \right|^2 \right)^{1/2}.
\end{align*}
We integrate both sides of the inequality above {(which can be done by Lemma $\ref{measurablityLemma}$)}, and by H\"{o}lder's inequality, we obtain
\begin{align}\label{integral}
&\int \limsup_{N \to \infty}\sup_{\alpha, \beta \in \RR} \left| \frac{1}{N} \sum_{n=1}^N f_1(T^{an}x)f_2(T^{bn}x)e(\alpha n^2 + \beta n) \right|^2 d\mu(x)\\
&\leq \frac{2}{H} + 4\left( \frac{1}{H+1} \sum_{h=1}^H \int \limsup_{N \to \infty} \sup_{\alpha \in \RR} \left| \frac{1}{N} \sum_{n=1}^N (f_1 \cdot f_1 \circ T^{ah})(T^{an}x) (f_2 \cdot f_2 \circ T^{bh})(T^{bn}x) e(-2\alpha hn) \right|^2 d\mu(x) \right)^{1/2}. \nonumber
\end{align}
Note that the inside of the integral on the right hand side of $(\ref{integral})$ is a double recurrence Wiener-Wintner average (by setting $t = -2\alpha h$) for each $h$. By ($\ref{estimate}$), we have
\begin{align*}
&\int \limsup_{N \to \infty} \sup_{\alpha \in \RR} \left| \frac{1}{N} \sum_{n=1}^N (f_1 \cdot f_1 \circ T^{ah})(T^{an}x) (f_2 \cdot f_2 \circ T^{bh})(T^{bn}x) e(-2\alpha hn) \right|^2 d\mu(x) \\
&\lesssim_{a, b} \min \left\{ \vertiii{f_1 \cdot f_1 \circ T^{ah}}_3^2 \, ,\vertiii{f_2 \cdot f_2 \circ T^{bh}}_3^2 \right\}. \end{align*}
Therefore,
\begin{align*}
&\int \limsup_{N \to \infty}\sup_{\alpha, \beta \in \RR} \left| \frac{1}{N} \sum_{n=1}^N f_1(T^{an}x)f_2(T^{bn}x)e(\alpha n^2 + \beta n) \right|^2 d\mu(x)\\
&\lesssim_{a, b} \frac{1}{H} + \min \left\{\left( \frac{1}{H} \sum_{h=1}^H \vertiii{f_1 \cdot f_1 \circ T^{ah}}_3^2 \right)^{1/2} \, , \left( \frac{1}{H} \sum_{h=1}^H \vertiii{f_2 \cdot f_2 \circ T^{bh}}_3^2 \right)^{1/2}  \right\} \\
&\lesssim_{a, b} \frac{1}{H} + \min \left\{ \left(\frac{1}{H} \sum_{h=1}^H \vertiii{f_1 \cdot f_1 \circ T^{ah}}_3^4\right)^{1/4} \, , \left(\frac{1}{H} \sum_{h=1}^H \vertiii{f_2 \cdot f_2 \circ T^{bh}}_3^4\right)^{1/4}  \right\},
\end{align*}
and by letting $H \to \infty$, we obtain
\begin{equation}\label{deg2Estimate} \int \limsup_{N \to \infty}\sup_{\alpha, \beta \in \RR} \left| \frac{1}{N} \sum_{n=1}^N f_1(T^{an}x)f_2(T^{bn}x)e(\alpha n^2 + \beta n) \right|^2 d\mu(x) \lesssim_{a, b} \min \left\{ \vertiii{f_1}_4^2 \, , \vertiii{f_2}_4^2 \right\}. \end{equation}
Since either $f_1$ or $f_2$ belongs to $\mathcal{Z}_3^\perp$, either $\vertiii{f_1}_4$ or $\vertiii{f_2}_4$ equals $0$. This completes the proof for the case where $p(n) = \alpha n^2 + \beta n$.

\subsubsection{The proof for any positive integer $k$} One of the key inequalities (besides van der Corput's lemma) used for the case where $k = 2$ is $(\ref{deg2Estimate})$, where we controlled the integral of the $\limsup$ of the averages by an appropriate Gowers-Host-Kra seminorm. We generalize this inequality for polynomials with higher degree with induction (on $k$) and van der Corput's inequality.
\begin{lemma}\label{EstimationLemmaPolyUnif}
	Let $(X, \mathcal{F}, \mu, T)$ be an ergodic system, and $f_1, f_2 \in L^\infty$. Then
	\begin{equation}\label{EstimationPolyUnif} \int \limsup_{N \to \infty} \sup_{p \in \RR_k[\xi]} \left| \frac{1}{N} \sum_{n=1}^N f_1(T^{an}x)f_2(T^{bn}x)e(p(n)) \right|^2 d\mu(x) \lesssim_{a, b, k} \min \left\{ \vertiii{f_1}_{k+2}^2 \, , \vertiii{f_2}_{k+2}^2 \right\}. \end{equation}
\end{lemma}
\begin{proof}[Proof of Lemma $\ref{EstimationLemmaPolyUnif}$]
	We proceed by induction on $k$. The base case $k = 1$ is clear from Theorem $\ref{Thm-estimate}$. Now suppose the claim holds for $k = 1, 2, \ldots, l$. Let $p(n)$ be a polynomial with degree $l+1$. If $q_h(n) = p(n+h)-p(n)$, then $q_h(n)$ is a polynomial of degree less than or equal to $l$ for all $h$, viewing $n$ as the variable. By van der Corput's lemma and the Cauchy-Schwarz inequality, we know that
	\begin{align*}
	& \limsup_{N \to \infty}\sup_{p \in \RR_{l+1}[t]} \left| \frac{1}{N} \sum_{n=1}^N f_1(T^{an}x)f_2(T^{bn}x) e(p(n)) \right|^2 \\
	&\leq \frac{2}{H+1} + \frac{4}{H+1} \sum_{h=1}^H  \limsup_{N \to \infty}\sup_{q_h \in \RR_l[t]} \left| \frac{1}{N} \sum_{n=1}^{N- h - 1} (f_1\cdot f_1 \circ T^{ah})(T^{an}x)(f_2\cdot f_2 \circ T^{bh})(T^{bn}x) e(q_h(n)) \right| \\
	&\leq \frac{2}{H+1} + 4\left(\frac{1}{H+1} \sum_{h=1}^H  \limsup_{N \to \infty}\sup_{q_h \in \RR_l[t]} \left| \frac{1}{N} \sum_{n=1}^{N- h - 1} (f_1\cdot f_1 \circ T^{ah})(T^{an}x)(f_2\cdot f_2 \circ T^{bh})(T^{bn}x) e(q_h(n)) \right|^2 \right)^{1/2} .
	\end{align*}
	By integrating both sides {(which is possible by Lemma $\ref{measurablityLemma}$)} and applying H\"{o}lder's inequality, we have
	\begin{align*}
	& \int \limsup_{N \to \infty}\sup_{p \in \RR_{l+1}[t]} \left| \frac{1}{N} \sum_{n=1}^N f_1(T^{an}x)f_2(T^{bn}x) e(p(n)) \right|^2 d\mu(x) \\
	&\leq \frac{2}{H+1} + 4\left(\frac{1}{H+1} \sum_{h=1}^H  \int \limsup_{N \to \infty} \sup_{q_h \in \RR_l[t]} \left| \frac{1}{N} \sum_{n=1}^{N- h - 1} (f_1\cdot f_1 \circ T^{ah})(T^{an}x)(f_2\cdot f_2 \circ T^{bh})(T^{bn}x) e(q_h(n)) \right|^2 d\mu \right)^{1/2}.
	\end{align*}
	For any $1 \leq h \leq H$, the inductive hypothesis tells us that
	\begin{align*}
	&\int \limsup_{N \to \infty} \sup_{q_h \in \RR_{l}[t]} \left| \frac{1}{N} \sum_{n=1}^{N- h - 1} (f_1\cdot f_1 \circ T^{ah})(T^{an}x)(f_2\cdot f_2 \circ T^{bh})(T^{bn}x) e(q_h(n)) \right|^2 d\mu \\
	&\lesssim_{a, b, l} \min \left\{ \vertiii{f_1 \cdot f_1 \circ T^{ah}}_{l+2}^{2}, \vertiii{f_2 \cdot f_2\circ T^{bh}}_{l+2}^{2} \right\}. 
	\end{align*}
	Therefore,
	\begin{align*}
	&\int \limsup_{N \to \infty}\sup_{p \in \RR_{l+1}[t]} \left| \frac{1}{N} \sum_{n=1}^N f_1(T^{an}x)f_2(T^{bn}x) e(p(n)) \right|^2 d\mu(x) \\
	&\lesssim_{a, b, l} \frac{1}{H} + \min \left\{ \left( \frac{1}{H} \sum_{h=1}^H \vertiii{f_1 \cdot f_1 \circ T^{ah}}_{l+2}^{2} \right)^{1/2}, \left( \frac{1}{H} \sum_{h=1}^H \vertiii{f_2 \cdot f_2 \circ T^{bh}}_{l+2}^{2} \right)^{1/2} \right\} \\
	&\lesssim_{a, b, l} \frac{1}{H} + \min \left\{ \left( \frac{1}{H} \sum_{h=1}^H \vertiii{f_1 \cdot f_1 \circ T^{ah}}_{l+2}^{2^{l+2}} \right)^{2^{-(l+2)}}, \left( \frac{1}{H} \sum_{h=1}^H \vertiii{f_2 \cdot f_2 \circ T^{bh}}_{l+2}^{2^{l+2}} \right)^{2^{-(l+2)}} \right\},
	\end{align*}
	and if we let $H \to \infty$, we obtain
	\begin{equation*} \int \limsup_{N \to \infty}\sup_{p \in \RR_{l+1}[t]} \left| \frac{1}{N} \sum_{n=1}^N f_1(T^{an}x)f_2(T^{bn}x) e(p(n)) \right|^2 d\mu(x) \lesssim_{a, b, l} \min \left\{ \vertiii{f_1}_{l+3}^{2}, \vertiii{f_2}_{l+3}^{2}\right\}.\end{equation*}
\end{proof}

\begin{proof}[Proof of (1) of Theorem $\ref{Main Theorem}$]
By our assumption, either $f_1$ or $f_2$ belongs to $\mathcal{Z}_{k+1}^\perp$, which implies that either $\vertiii{f_1}_{k+2}$ or $\vertiii{f_2}_{k+2}$ equals $0$, hence the right hand side of the inequality $(\ref{EstimationPolyUnif})$ equals $0$. Thus, there exists a set of full measure $X_{f_1, f_2}$ such that for any $x \in X_{f_1, f_2}$ and $p \in \RR_k[\xi]$, we have
\[ \limsup_{N \to \infty} \sup_{p \in \RR_k[\xi]} \left| \frac{1}{N} \sum_{n=1}^N f_1(T^{an}x)f_2(T^{bn}x) e(p(n)) \right| = 0. \]
\end{proof}

\subsection{Proofs of (2) and (3) of Theorem $\ref{Main Theorem}$}
In this section, we first prove (2) of Theorem $\ref{Main Theorem}$. Then we use this, together with (1), to prove (3) of the same theorem.

First, we prove the following approximation lemma; this allows us to reduce our proof to the case where $f_1$ and $f_2$ are both continuous functions on an ergodic nilsystem. The following inequality will be useful when dominating the averages in norm: Given a measure-preserving system $(X, \mathcal{F}, \mu, T)$ and $F \in L^\alpha(\mu)$ for $\alpha \in (1, \infty)$, we have
\begin{equation}
\label{maxIneq} \norm{\sup_N \frac{1}{N} \sum_{n=1}^N F(T^nx)}_\alpha \leq \frac{\alpha}{\alpha-1} \norm{F}_\alpha.
\end{equation}
This inequality can be obtained by using the maximal ergodic theorem (see, for example, \cite[Theorem 1.8]{AssaniWWET} for a proof).
\begin{lemma}\label{approximation}
	Let $(X, \mathcal{F}, \mu, T)$ be a measure-preserving system, and $a$ and $b$ be distinct integers. Let  $f_1, f_2 \in L^\infty(\mu)$. Suppose there exist two sequences of functions $(f_1^i)_i$ and $(f_2^i)_i$ in $L^\infty(\mu)$ such that $\norm{f_1^i}_{L^\infty(\mu)} < M$ for some constant $M > 0$ for any $i \in \NN$, $f_j^i \to f_j$ in $L^2(\mu)$-norm as $i \to \infty$ for each $j = 1, 2$, and for each $i$, there exists a set of full measure $X_i$ such that for any $x \in X_i$ and any $p \in \RR_k[\xi]$ for each $k \in \NN$, the averages
		\[ \frac{1}{N} \sum_{n=1}^N f_1^i(T^{an}x)f_2^i(T^{bn}x) e(p(n)) \]
		converge. Then there exists a set of full measure $X_{\infty} \subset X$ such that for any $x \in X_\infty$ and any $p \in \RR_k[\xi]$, the averages
		\[ \frac{1}{N} \sum_{n=1}^N f_1(T^{an}x)f_2(T^{bn}x) e(p(n))\]
		converge for each $k \in \NN$.
\end{lemma}
\begin{proof}
	For each $j = 1, 2$, we can write $f_j = (f_j - f_j^i) + f_j^i$ for each $i$, so we can rewrite the averages as follows:
	\begin{align}\label{4terms}
	 W_N(f_1, f_2, x, p) 
	&= \frac{1}{N} \sum_{n=1}^N f_1(T^{an}x)f_2(T^{bn}x)e(p(n)) \nonumber \\
	&= \frac{1}{N} \sum_{n=1}^N (f_1 - f_1^i)(T^{an}x)f_2(T^{bn}x)e(p(n)) + \frac{1}{N} \sum_{n=1}^Nf_1^i(T^{an}x)(f_2 - f_2^i)(T^{bn}x)e(p(n))\\
	&+ \frac{1}{N} \sum_{n=1}^Nf_1^i(T^{an}x)f_2^i(T^{bn}x)e(p(n)). \nonumber
	\end{align}
	Ultimately, we would like to show that there exists a set of full measure  $X_\infty \subset X$ such that for any $x \in X_\infty$,
	\begin{equation}\label{sup-inf} \mathcal{L}_R(W_N(f_1, f_2, x, p)) = 
	\sup_{p \in \RR_k[\xi]} \left( \limsup_{N \to \infty} Re (W_N(f_1, f_2, x, p)) - \liminf_{N \to \infty} Re (W_N(f_1, f_2, x, p)) \right) = 0, \end{equation}
	and
	\begin{equation}\label{sup-inf-Im} \mathcal{L}_I(W_N(f_1, f_2, x, p)) = 
	\sup_{p \in \RR_k[\xi]} \left( \limsup_{N \to \infty} Im (W_N(f_1, f_2, x, p)) - \liminf_{N \to \infty} Im (W_N(f_1, f_2, x, p)) \right) = 0. \end{equation}
	To show $(\ref{sup-inf})$, we first note that the third term on the right hand side of $(\ref{4terms})$ vanishes for $\mu$-a.e. $x \in X$ for each $i$ after applying $\mathcal{L}_R$ since we know that the averages converge for all $x \in \bigcap_{i=1}^\infty X_i$, which is a set of full measure, and for any $p \in \RR_k[\xi]$ and $i \in \NN$. To show the remaining terms vanish, we apply H\"{o}lder's inequality as well as the inequality $(\ref{maxIneq})$. For instance, for the first term of $(\ref{4terms})$, we have that
	\begin{align*} 
	\sup_{p \in \RR_k[\xi]}\left| \frac{1}{N} \sum_{n=1}^N(f_1 - f_1^i)(T^{an}x)f_2(T^{bn}x)e(p(n)) \right| 
	&\leq \frac{1}{N} \sum_{n=1}^N \left|(f_1 - f_1^i)(T^{an}x)f_2(T^{bn}x)\right| \\
	&\leq \norm{f_2}_{L^\infty(\mu)}\frac{1}{N} \sum_{n=1}^N\left|f_1 - f_1^i\right|(T^{an}x) \, .
	\end{align*}
	If we take supremum over $N$ on both sides, we would have
	\begin{align*} &\sup_{N \geq 1} \sup_{p \in \RR_k[\xi]}\left| \frac{1}{N} \sum_{n=1}^N(f_1 - f_1^i)(T^{an}x)f_2(T^{bn}x)e(p(n)) \right| \leq  \norm{f_2}_{L^\infty(\mu)} \left(\sup_{N \geq 1}  \frac{1}{N} \sum_{n=1}^N\left|f_1 - f_1^i\right|(T^{an}x) \right) , \end{align*}
	so we integrate both sides {(which is possible by Lemma $\ref{measurablityLemma}$)} and apply H\"{o}lder's inequality for the right hand side to obtain
	\begin{align*}
	&\int \sup_{N \geq 1} \sup_{p \in \RR_k[\xi]}\left|\frac{1}{N} \sum_{n=1}^N(f_1 - f_1^i)(T^{an}x)f_2(T^{bn}x)e(p(n)) \right| d\mu(x) \leq \norm{f_2}_{L^\infty(\mu)} \norm{\sup_{N \geq 1}  \frac{1}{N} \sum_{n=1}^N \left|f_1 - f_1^i\right|(T^{an}x)}_{L^2(\mu)}   \nonumber \, .
	\end{align*}
We apply the inequality $(\ref{maxIneq})$ for the case where $\alpha = 2$ to the $L^2(\mu)$-norm on the right hand side to obtain
\begin{align}\label{secondTerm}
&\int \sup_{N \geq 1} \sup_{p \in \RR_k[\xi]}\left|\frac{1}{N} \sum_{n=1}^N(f_1 - f_1^i)(T^{an}x)f_2(T^{bn}x)e(p(n)) \right| d\mu(x) \leq 2\norm{f_2}_{L^\infty(\mu)}\norm{f_1-f_1^i}_{L^2(\mu)}\, .
\end{align}

 By the similar argument as in the first term of $(\ref{4terms})$ (and recalling that $\norm{f_1^i}_{L^\infty(\mu)} \leq M$ for all $i$), we can also obtain an estimate for the second term:
	\begin{equation}\label{thirdTerm} \int \sup_{N \geq 1} \sup_{p \in \RR_k[\xi]}\left| \frac{1}{N} \sum_{n=1}^{N} f_1^i(T^{an}x)(f_2-f_2^i)(T^{bn}x)e(p(n)) \right| d\mu(x) \leq 2M\|f_2 - f_2^i\|_{L^2(\mu)}\, . \end{equation}
	We are now ready to verify $(\ref{sup-inf})$. We note that
	\begin{align*}
	0 &\leq \sup_{p \in \RR_k[\xi]} \left(\limsup_{N \to \infty} Re\left( W_N(f_1, f_2, x, p)\right) - \liminf_{N \to \infty} Re \left(W_N(f_1, f_2, x, p)\right) \right) \nonumber\\
	&\leq 2\liminf_{i \to \infty} \left(\sup_{N \geq 1} \sup_{p \in \RR_k[\xi]}\left| \frac{1}{N} \sum_{n=1}^N (f_1-f_1^i)(T^{an}x)f_2(T^{bn}x) e(p(n))\right| \right. \\
	&+ \left. \sup_{N \geq 1} \sup_{p \in \RR_k[\xi]}\left| \frac{1}{N} \sum_{n=1}^N f_1^i(T^{an}x)(f_2 - f_2^i)(T^{bn}x) e(p(n))\right| \right).
	\end{align*}
	According to the inequalities $ (\ref{secondTerm})$ and $(\ref{thirdTerm})$, the integral of each average in the right-hand side of the inequality above is bounded by a constant multiple of either $\| f_1 - f_1^i \|_{L^2(\mu)}$ or $\| f_2 - f_2^i \|_{L^2(\mu)}$. These norms converge to $0$ as $i \to \infty$. Using those inequalities together with Fatou's lemma, we obtain
	\begin{align}
	& \int \sup_{p \in \RR_k[\xi]} \left(\limsup_{N \to \infty} Re\left( W_N(f_1, f_2, x, p)\right) - \liminf_{N \to \infty} Re \left(W_N(f_1, f_2, x, p)\right) \right) d\mu \label{LHS}\\
	&\leq 2\liminf_{i \to \infty}\left(\int \sup_{N \geq 1} \sup_{p \in \RR_k[\xi]}\left| \frac{1}{N} \sum_{n=1}^N (f_1-f_1^i)(T^{an}x)f_2(T^{bn}x) e(p(n))\right| d\mu(x) \right. \nonumber \\
	&+ \left. \int \sup_{N \geq 1} \sup_{p \in \RR_k[\xi]}\left| \frac{1}{N} \sum_{n=1}^N f_1^i(T^{an}x)(f_2 - f_2^i)(T^{bn}x) e(p(n))\right| d\mu(x) \right) = 0. \nonumber
	\end{align}
	Since inside the integral of ($\ref{LHS}$) is nonnegative, $(\ref{sup-inf})$ is established for $\mu$-a.e. $x \in X$. Therefore, there exists a set of full measure $X_{R} \subset \bigcap_{i=1}^\infty X_i$ such that for any $x \in X_{R}$ and any polynomial $p \in \RR_k[\xi]$, the real part of the sequence $(W_N(f_1, f_2, x, p))_N$ converge for any $k \in \NN$.
	
	Similarly, we can show that $(\ref{sup-inf-Im})$ holds for $\mu$-a.e. $x \in \bigcap_{i=1}^\infty X_i$, so there exists a set of full measure $X_{I} \subset \bigcap_{i=1}^\infty X_i$ such that for any $x \in X_{I}$ and any polynomial $p \in \RR_k[\xi]$, the imaginary part of the sequence $(W_N(f_1, f_2, x, p))_N$ converge for any $k \in \NN$. So if we set $X_\infty = X_{R} \cap X_{I}$, we obtain the desired set of full measure. \end{proof}
To prove (2) of Theorem $\ref{Main Theorem}$, we first prove this for the case where $(X, \mathcal{F}, \mu, T)$ is an ergodic nilsystem, and $f_1$ and $f_2$ are both continuous functions on $X$; under these assumptions, the averages converge \textit{for all} $x \in X$. 
The key ingredient of this proof is Leibman's pointwise convergence theorem of polynomial actions on a nilsystem \cite{Leibman}, which is used to prove the following lemma.
\begin{lemma}\label{BothInZk+1}
	{Let $(X, \mathcal{F}, \mu, T)$ be a {$(k+1)$-step ergodic nilsystem. Suppose $f_1, f_2 \in \mathcal{C}(X)$}, and $a, b \in \ZZ$ such that $a \neq b$. Then for any $x \in X$ and $p \in \RR_k[\xi]$, the averages 
	\begin{equation}\label{generalCV} \frac{1}{N} \sum_{n=1}^N f_1(T^{an}x)f_2(T^{bn}x)e(p(n)) \end{equation}
	converge as $N \to \infty$.}
\end{lemma}
\begin{proof}[Proof of Lemma $\ref{BothInZk+1}$] 
{{Let $t$ be any real number}. Suppose first that we fix a polynomial $q \in \RR[\xi]$. Suppose also that $X$ is a $(k+1)$-step nilsystem. Since we know that $\TT = \RR/\ZZ$ is a one-step nilmanifold, the system $(\TT, \mathcal{B}, m, R_t)$ is a nilsystem, where $\mathcal{B}$ is a Borel $\sigma$-algebra of $\TT$, $m$ is the usual Borel probability measure, and $R_t$ is a rotation by $t$ i.e. for any $\alpha \in (0, 1]$, we have $R_t(e(\alpha)) = e(\alpha + t)$. Thus, $(X^2 \times \TT, \mathcal{F}^2 \otimes \mathcal{B}, \mu^2 \otimes m)$ is a $k+1$-step nilmanifold. Suppose $F: X^2 \times \TT \to \CC$ for which
\[ F(x_1, x_2, e(\alpha)) = f_1(x_1)f_2(x_2)e(\alpha). \]
Then $F$ is continuous on $X^2 \times \TT$. 
Hence, for any $\alpha \in [0, 1)$, we see that
\[ \frac{1}{N} \sum_{n=1}^N F(T^{an}x_1, T^{bn}x_2, R_t^{{q(n)}}e(\alpha)) = \frac{e(\alpha)}{N} \sum_{n=1}^N f_1(T^{an}x_1)f_2(T^{bn}x_2) e({q(n)}t). \]
{Note that the left-hand side of the equation above converges by Leibman's convergence result for all $(x_1, x_2, e(\alpha)) \in X^2 \times \TT$ \cite[Theorems A, B]{Leibman}\footnote{In Leibman's paper, the polynomial sequences are defined for polynomials with integer coefficients. However, the cited theorems are proven for the case where one has a polynomial with real coefficients, as it is mentioned in \cite[\S3.13]{Leibman}, provided that this action makes sense. Since the element of the group corresponding to the rotation in $\TT$ belongs to the identity component of the group, a real polynomial exponential makes sense.}, so the averages in ($\ref{generalCV}$) converge for \textit{all} $(x_1, x_2, e(\alpha)) \in X^2 \times \TT$ for any $t \in \RR$}. In particular, it converges when $x_1 = x_2 = x$ and $\alpha = 0$, so we have shown that for any $x \in X$, the averages
\begin{equation}\label{AvgWitht} \frac{1}{N} \sum_{n=1}^N f_1(T^{an}x)f_2(T^{bn}x)e(tq(n)) \end{equation}
converge as $N \to \infty$.

Now we fix $t \in (0, 1)$. Then for any $k$-th degree polynomial $p$, there exists another $k$-th degree polynomial $q$ such that $p = tq$ (e.g. if $\displaystyle{p(n) = \sum_{l=0}^k c_ln^l}$, then we can set $\displaystyle{ q(n) = \sum_{l=0}^k c'_ln^l}$, where $c'_l = c_l/t$). Thus, we have
\[ \frac{1}{N} \sum_{n=1}^N f_1(T^{an}x)f_2(T^{bn}x)e(p(n)) = \frac{1}{N} \sum_{n=1}^N f_1(T^{an}x)f_2(T^{bn}x)e(tq(n)),  \]
and we know that the averages in the right hand side is $(\ref{AvgWitht})$, so the averages on the left hand side converge for all $x \in X$ as $N \to \infty$ . Hence, we have shown that the claim holds for the case where $f_1$ and $f_2$ are both continuous functions on an ergodic nilsystem.}
\end{proof}
Using the preceding lemma as well as Lemma $\ref{approximation}$, we prove (2) of Theorem $\ref{Main Theorem}$.
\begin{proof}[Proof of $(2)$ of Theorem $\ref{Main Theorem}$]
 By the structure theorem \cite[Theorem 10.1]{HostKraNEA}, the $k+1$-th Host-Kra-Ziegler factor is the inverse limit of a sequence of $k+1$-step ergodic nilsystems that are factor of $(X, \mathcal{F}, \mu, T)$. This implies that if $f_1 \in \mathcal{Z}_{k+1}$, then for any $k+1$-step ergodic nilsystem $(N, \mathcal{N}, \mu, T)$ that is a factor of $(X, \mathcal{F}, \mu, T)$, we have $\norm{\expec(f_1 | \mathcal{N})}_{L^\infty(\mu)} \leq \norm{f_1}_{L^\infty(\mu)}$. Hence, by Lemma $\ref{approximation}$, it suffices to show that the statement of the theorem holds for the case where $f_1$ and $f_2$ are bounded and measurable with respect to a $k+1$-step ergodic nilsystem $(N, \mathcal{N}, \mu, T)$ that is a factor of $(X, \mathcal{F}, \mu, T)$. But since we know that if $g_1$ and $g_2$ are continuous functions on $N$, then the averages
 \[ \frac{1}{N} \sum_{n=1}^N g_1(T^{an}x)g_2(T^{bn}x)e(p(n)) \]
 converge for all $x \in N$ and $p \in \RR_k[\xi]$ by Lemma $\ref{BothInZk+1}$. Furthermore, by density, there exist sequences of continuous functions $(\tilde{g}_1^i)_i$ and $(\tilde{g}_2^i)_i$ on $N$ such that $\tilde{g}_j^i \to f_j$ in $L^2(\mu)$ as $i \to \infty$ for each $j = 1, 2$. We can construct another sequence of continuous functions $(g_1^i)_i$ such that
 \[ g_1^i(x) = \left\{ \begin{array}{ll}
 \min\left(\tilde{g}_1^i(x), \norm{f_1}_{L^\infty(\mu)}\right) & \text{if } \tilde{g}_1^i(x) \geq 0, \\
 \max\left(\tilde{g}_1^i(x), -\norm{f_1}_{L^\infty(\mu)}\right) & \text{if } \tilde{g}_1^i(x) < 0, \\
 \end{array} \right. \]
 so that $g_1^i \to f_1$ in $L^2(\mu)$ as $i \to \infty$, and $\norm{g_1^i}_{L^\infty(\mu)} < \norm{f_1}_{L^\infty(\mu)}$ for each $i \in \NN$. Thus, we can apply Lemma $\ref{approximation}$ again (for the sequences $(g_1^i)_i$ and $(\tilde{g}_2^i)_i$) to show that the statement of the theorem holds for the case where $f_1$ and $f_2$ are bounded and measurable functions on $N$.
\end{proof}
Now we are ready to prove (3) of Theorem $\ref{Main Theorem}$ using (1) and (2).
\begin{proof}[Proof of $(3)$ of Theorem $\ref{Main Theorem}$]
Since any continuous function $\phi$ on $\TT$ can be approximated by a linear combination of complex trigonometric functions, it suffices to prove this claim by showing that the averages
\begin{equation}\label{ultimate} \frac{1}{N} \sum_{n=1}^N f_1(T^{an}x)f_2(T^{bn}x)e(p(n))\end{equation}
converge off a single null-set independent of $p$. First we find a single set of full measure independent of $p \in \RR_k[\xi]$ for which the averages converge. For each $j = 1, 2$, we write $f_j = \expec(f_j| \mathcal{Z}_{k+1}) + f_j^\perp$, where $f_j^\perp \in \mathcal{Z}_{k+1}^\perp$. By (1) of Theorem $\ref{Main Theorem}$, we merely need to show that the averages
\begin{equation}\label{projection} \frac{1}{N} \sum_{n=1}^N \expec(f_1| \mathcal{Z}_{k+1})(T^{an}x)\expec(f_2| \mathcal{Z}_{k+1})(T^{bn}x)e(p(n)) \end{equation}
converge on a set of full measure independent of $p \in \RR_k[\xi]$. By (2) of Theorem $\ref{Main Theorem}$, we know that there exists a set of full measure $X_{f_1, f_2, k}$ such that for any $x \in X_{f_1, f_2, k}$, the averages in $(\ref{projection})$ converges for all $p \in \RR_k[\xi]$. Thus, if we set
\[ X_{f_1, f_2} = \bigcap_{k=1}^\infty X_{f_1, f_2, k}, \]
then $X_{f_1, f_2}$ is a set of full measure independent for which the averages in $(\ref{ultimate})$ converge for all $x \in X_{f_1, f_2}$ and for all polynomials $p$ with real coefficients.
\end{proof}

\section{Corollaries}\label{sec:Corollaries}
One can directly prove the following weaker version of Theorem $\ref{Main Theorem}$. We originally proved this result before we obtained the proof of Theorem $\ref{Main Theorem}$, and used different approaches to prove a certain matter (e.g. Anzai's skew-product transformation on $\TT$ \cite{Anzai}). See \cite{WWDR_nil} for more detail.
\begin{cor}\label{WWDR_nil}
	Let $(X, \mathcal{F}, \mu, T)$ be a standard ergodic dynamical system, $a, b \in \ZZ$ such that $a \neq b$, and $f_1, f_2 \in L^2(X)$. Suppose $p(n)$ is a degree-$k$ polynomial with real coefficients, where $k \geq 1$. Let 
	\[W_N(f_1, f_2,x ,p, t) = \frac{1}{N} \sum_{n=1}^{N} f_1(T^{an}x)f_2(T^{bn}x) e^{2\pi i p(n) t}. \]
	\begin{enumerate}
		\item If either $f_1$ or $f_2$ belongs to $\mathcal{Z}_{k+1}^\perp$, then there exists a set of full measure $X_{f_1, f_2, p}$ such that for all $x \in X_{f_1, f_2, p}$,
		\begin{equation*}\label{UniformDWWP} \limsup_{N \to \infty} \sup_{t \in \RR} \left|W_N(f_1, f_2, x,p, t) \right| = 0. \end{equation*}
		\item If $f_1, f_2 \in \mathcal{Z}_{k+1}$, then there exists a set of full measure $X_{f_1, f_2, p}$ such that for all $x \in X_{f_1, f_2, p}$, the averages $W_N(f_1, f_2, x, p, t)$ converge for all $t \in \RR$.
	\end{enumerate}
\end{cor}
\begin{proof}
If $p$ is a $k$-th degree polynomial with real coefficients, then so is $tp$, so both (1) and (2) hold if $x$ belongs to the set of full measure obtained in Theorem $\ref{Main Theorem}$.
\end{proof}
The direct consequence of Corollary $\ref{WWDR_nil}$ is that the sequence $u_n = f_1(T^{an}x)f_2(T^{bn}x)$ is a universally good weight for the polynomial return time averages in norm.
\begin{cor}
Let $(X, \mathcal{F}, \mu, T)$ be an ergodic system, $a, b \in \ZZ$ such that $a \neq b$. Then there exists a set of full measure $\tilde{X}$ such that for any $x \in \tilde{X}$, a polynomial $p \in \ZZ[\xi]$, and for any measure-preserving system $(Y, \mathcal{G}, \nu, S)$ and $g \in L^\infty(\nu)$, the averages
\[ \frac{1}{N} \sum_{n=1}^N f_1(T^{an}x)f_2(T^{bn}x)g \circ S^{p(n)} \]
converge in $L^2(\nu)$.
\end{cor}
\begin{proof}
This is a direct application of Theorem $\ref{Main Theorem}$ and the spectral theorem.
\end{proof}
\section*{Acknowledgment}
We thank Benjamin Weiss for suggesting to look at the extension of the double recurrence Wiener-Wintner result to nilsequences during the 2014 Ergodic Theory Workshop at UNC Chapel Hill. We also thank El Houcein El Abdalaoui for his interest in this problem. Lastly, we thank the anonymous referee for his/her suggestions and comments.

\bibliographystyle{plain}
\bibliography{RM_Bib}

\end{document}